\documentclass[12pt,reqno]{amsart}
\usepackage{amsmath,amsfonts,amsthm,amssymb,amsxtra,bbm,enumerate}
\usepackage[a4paper]{geometry}
\newtheorem{theorem}{Theorem}[section]

\newtheorem{lemma}[theorem]{Lemma}
\newtheorem{corollary}[theorem]{Corollary}

\theoremstyle{definition}

\newtheorem{assumption}[theorem]{Assumption}

\theoremstyle{remark}

\newtheorem{remark}[theorem]{Remark}
\newtheorem*{remark*}{Remark}

\numberwithin{equation}{section}

\newcommand{\C}{\mathbb{C}}

\newcommand{\R}{\mathbb{R}}

\newcommand{\1}{\mathbbm{1}}

\DeclareMathOperator{\im}{Im}

\DeclareMathOperator{\re}{Re}
\DeclareMathOperator{\spec}{spec}

\DeclareMathOperator*{\esssup}{ess\,sup}

\DeclareMathOperator{\Tr}{Tr}

\DeclareMathOperator{\tr}{Tr}

\newcommand{\abs}[1]{\lvert#1\rvert}

\newcommand{\norm}[1]{\lVert#1\rVert}
\newcommand{\Norm}[1]{\left\lVert#1\right\rVert}

\newcommand{\wt}{\widetilde} 
\newcommand{\eps}{\varepsilon}

\begin{document}

\title[Density of a product of spectral projections]{The spectral density of a product \\ of spectral projections}

\author{Rupert L. Frank}
\address{Rupert L. Frank, Mathematics 253-37, Caltech, Pasadena, CA 91125, USA}
\email{rlfrank@caltech.edu}

\author{Alexander Pushnitski}
\address{Alexander Pushnitski, Department of Mathematics, King's College London, Strand, London, WC2R 2LS, UK}
\email{alexander.pushnitski@kcl.ac.uk}

\begin{abstract}
We consider the product of spectral projections 
$$
\Pi_\eps(\lambda)
=
{\mathbbm{1}}_{(-\infty,\lambda-\eps)}(H_0)
{\mathbbm{1}}_{(\lambda+\eps,\infty)}(H)
{\mathbbm{1}}_{(-\infty,\lambda-\eps)}(H_0)
$$
where $H_0$ and $H$ are the free and the perturbed Schr\"odinger operators with a short
range potential, $\lambda>0$ is fixed and $\eps\to0$. 
We compute the leading term of the asymptotics of $\tr f(\Pi_\eps(\lambda))$ 
as $\eps\to0$ for continuous functions $f$ vanishing sufficiently fast near zero. 
Our construction elucidates calculations that appeared earlier in the theory 
of ``Anderson's orthogonality catastrophe'' and emphasizes the role of Hankel
operators in this phenomenon.
\end{abstract}

%\subjclass[2000]{Primary 35P15, Secondary 35J10, 47F05}

\maketitle

%\renewcommand{\thefootnote}{${}$} \footnotetext{\copyright\, 2014 by the authors. This paper may be reproduced, in its entirety, for non-commercial purposes.}

%%%%%%%%%%%%%%%%%%%%%%%%%%%%%%
\section{Introduction}\label{sec:0}
%%%%%%%%%%%%%%%%%%%%%%%%%%%%%%

\subsection{Motivation from mathematical physics}
This paper is partly motivated by a phenomenon called ``Anderson's orthogonality catastrophe'', 
which has been intensively discussed in the physics literature and has recently attracted attention from a mathematical perspective; see \cite{GKM,KOS,GKMO} and the literature cited therein. Let $H_0$ and $H$ be the free and the perturbed Schr\"odinger operators 
$$
H_0=-\Delta\,, \qquad H=-\Delta+V \qquad
\text{ in } L^2(\R^d)\,, \quad d\geq1\,,
$$
where, for the sake of simplicity, the real-valued potential $V$ is assumed to be bounded and compactly supported. 
For $\lambda>0$, consider the product of spectral projections
\begin{equation}
\Pi(\lambda)
=
\1_{(-\infty,\lambda)}(H_0)
\1_{(\lambda,\infty)}(H)
\1_{(-\infty,\lambda)}(H_0)
\quad
\text{ in } L^2(\R^d) \,. 
\label{eq:0d}
\end{equation}
One is interested in 
regularised versions of $\Pi(\lambda)$, obtained by  replacing the step functions 
$\1_{(-\infty,\lambda)}$, $\1_{(\lambda,\infty)}$  by functions with disjoint supports. 
More precisely, we consider two types of regularisations of $\Pi(\lambda)$,
\begin{equation}
\Pi_\eps^{(1)}(\lambda)
=
\1_{(-\infty,\lambda-\eps)}(H_0)
\1_{(\lambda+\eps,\infty)}(H)
\1_{(-\infty,\lambda-\eps)}(H_0)
\label{eq:0f}
\end{equation}
and 
\begin{equation}
\Pi_\eps^{(2)}(\lambda)
=
\psi_\eps^{-}(H_0-\lambda)
\psi_\eps^{+}(H-\lambda)
\psi_\eps^{-}(H_0-\lambda) \,,
\label{eq:0g}
\end{equation}
where $\psi_\eps^{\pm}$ are continuous functions on $\R$ that satisfy $0\leq \psi_\eps^{\pm}\leq 1$ and 
\begin{equation}
\psi_\eps^{+}(x)=
\begin{cases}
0 & \text{if}\ x\leq \eps\,,
\\
1 & \text{if}\ x\geq 2\eps\,,
\end{cases}
\qquad
\psi_\eps^{-}(x)=
\begin{cases}
1 & \text{if}\ x\leq -2\eps\,,
\\
0 & \text{if}\ x\geq -\eps\,.
\end{cases}
\label{eq:0h}
\end{equation}
It is not difficult to see that the operators $\Pi_\eps^{(j)}(\lambda)$, $j=1,2$, are trace class. On the other hand, $\Pi(\lambda)$ is typically not trace class and not even compact.  
We discuss the asymptotics 
of traces
\begin{equation}
\Tr f(\Pi^{(j)}_\eps(\lambda))\,, \qquad \eps\to0\,,\quad j=1,2 \,,
\label{eq:0e}
\end{equation}
where $f=f(t)$ is a continuous function which vanishes sufficiently fast as $t\to0$. 

It turns out that the asymptotics of the traces \eqref{eq:0e} is given 
in terms of the scattering matrix $S(\lambda)$ for the pair $H_0$, $H$ at energy $\lambda$. 
Let $\{ e^{i\theta_\ell(\lambda)}\}_{\ell=1}^L$, 
be the eigenvalues of $S(\lambda)$, enumerated with multiplicities taken into account. 
The scattering matrix is an operator in $L^2(\mathbb S^{d-1})$ for $d\geq2$ and is a 
$2\times 2$ matrix for $d=1$; thus, $L=\infty$ for $d\geq2$ and $L=2$ for $d=1$. 
Denote
\begin{equation}
\label{eq:4}
a_\ell(\lambda) = \frac12 \left| e^{i\theta_\ell(\lambda)}-1\right| 
= 
\left|\sin \frac{\theta_\ell(\lambda)}{2} \right| \in [0,1] \,,
\qquad \ell=1,\dots,L\,.
\end{equation}
For $j=1,2$ and for all continuous functions $f$, vanishing sufficiently fast at zero, we prove that 
\begin{equation}
\lim_{\eps\to0} |\ln \eps|^{-1}\Tr f(\Pi_\eps^{(j)}(\lambda))
=
\frac1{2\pi}\sum_{\ell=1}^L \int_{-\infty}^\infty 
f \biggl( \frac{a_\ell^2(\lambda)}{\cosh^2(\pi x)} \biggr) dx
=
\int_0^1 f(t)\mu_\lambda(t) \,dt\,, 
\label{eq:0i}
\end{equation}
where 
\begin{equation}
\mu_\lambda(t)
=
\frac1{2\pi^2}\sum_{\ell=1}^L
\frac{\1_{(0,a_\ell^2(\lambda))}(t)}{t\sqrt{1-(t/a_\ell^2(\lambda))}}. 
\label{eq:0ia}
\end{equation}
Of course, formula \eqref{eq:0ia} is obtained from the first equality in \eqref{eq:0i} 
by means of a change of variable.

In \cite{GKMO} a regularisation similar to $\Pi^{(2)}_\eps(\lambda)$ is used; the authors essentially prove \eqref{eq:0i} for $f(t)=t^n$. (They state it with $\geq$ instead of $=$, but in fact the proof contains the $\leq$ case as well.)
We believe that our construction is somewhat simpler than that of \cite{GKMO}; we replace
some of the heavy computations of \cite{GKMO} by ``soft'' operator theoretic arguments. 
In particular, our proof highlights the key role of Hankel operators here. 
We will say more about it in Section~\ref{sec:1.4}.

In fact, our proof of the asymptotics \eqref{eq:0i} uses very little specific information about 
Schr\"odinger operators. For this reason, we state it as a general operator theoretic result for a pair of self-adjoint operators $H_0$, $H$, satisfying some standard assumptions of scattering theory.

\subsection{Motivation from operator theory}
In \cite{P1,PYa1}, the spectral structure of the operator $\Pi(\lambda)$ (see \eqref{eq:0d}) was studied 
in detail (for pairs of operators $H_0$, $H$ satisfying some general assumptions of scattering theory).
In particular, it was proven that if $S(\lambda)\not=I$, then $\Pi(\lambda)$ has a non-trivial absolutely continuous 
spectrum, which consists of the union of intervals
\begin{equation}
\sigma_\text{ac}(\Pi(\lambda))
=
\bigcup_{\ell=1}^L \left[0,a_\ell^2(\lambda)\right];
\label{eq:0j}
\end{equation}
each interval contributes multiplicity one to the spectrum. Here $a_\ell(\lambda)$ are given by \eqref{eq:4}.

On the other hand, the regularisations $\Pi_\eps^{(j)}(\lambda)$, $j=1,2$, 
of $\Pi(\lambda)$ are compact operators (see Lemma~\ref{lem1}).
Thus, it is reasonable to ask how the transition from the compact operators  $\Pi_\eps^{(j)}(\lambda)$ 
to the operator $\Pi(\lambda)$ with non-trivial absolutely continuous spectrum occurs and how the eigenvalues of 
$\Pi_\eps^{(j)}(\lambda)$  concentrate to the spectral bands \eqref{eq:0j}. 
Formulas \eqref{eq:0i}, \eqref{eq:0ia} partially  answer this question: 
they give the eigenvalue density of $\Pi_\eps^{(j)}(\lambda)$ as $\eps\to0$ as 
an explicit function $\mu_\lambda(y)$. Note that $\mu_\lambda(y)$ is given as a sum over $\ell$, where 
each summand  
 is supported on a single band $\left[0,a_\ell^2(\lambda)\right]$.

\subsection{Notation}\label{sec:1.3}
We denote by $\mathbf S_p$, $p\geq1$, 
the standard Schatten class and by $\norm{\cdot}_{p}$ the norm in this class. 
$\mathbf B$ denotes the class of all bounded operators, $\mathbf S_\infty$ is the class
of all compact operators and $\norm{\cdot}$ is the operator norm. 
If $X$ and $Y$ are two normal operators (possibly between different Hilbert spaces) 
such that
$$
X|_{(\ker X)^\bot}
\quad\text{is unitarily equivalent to}\quad
Y|_{(\ker Y)^\bot}
$$
we write $X\approx Y$. 
It is well-known that $C^*C \approx C C^*$ for any bounded operator $C$.
We will frequently use the fact that the relation $X\approx Y$ implies $\tr f(X)=\tr f(Y)$ for all continuous functions $f$ with $f(0)=0$. For a set $\Omega\subset\R$, we denote by $\1_\Omega$ the characteristic function of this set.

%%%%%%%%%%%%%%%%%%%%%%%%%%%%%%
\section{Main result}\label{sec:1}
%%%%%%%%%%%%%%%%%%%%%%%%%%%%%%

\subsection{Assumptions}

Let $H_0$ and $H$ be self-adjoint, lower semi-bounded operators in a Hilbert space $\mathcal H$ such that
$$
H=H_0+V \,,
$$
where the perturbation $V$ admits a factorization of the form
$$
V=G^* V_0 G \,.
$$
Here, $\mathcal K$ is an auxiliary Hilbert space, $V_0$ is a bounded, self-adjoint operator in $\mathcal K$ and 
$G$ is a bounded operator from $\mathcal H$ to $\mathcal K$ satisfying
\begin{equation}
\label{eq:0a}
G(H_0+M)^{-1/2}\in\mathbf S_\infty
\end{equation}
for some constant $M>-\inf\spec H_0$. 
\begin{remark}
In fact, the boundedness of $G$ is not necessary for our construction; 
we state it as a  requirement here only in order to avoid inessential technical explanations. 
\end{remark}

We will need two assumptions: a global one (in spectral parameter) and a local one. 
The global assumption is

\begin{assumption}\label{ass0}
We have
\begin{equation}
\label{eq:0}
G(H_0+M)^{-1/2}\in \mathbf S_{\infty}, \quad
G(H_0+M)^{-1/2-m}\in\mathbf S_{2p}
\end{equation}
for some $p\geq1$ and $M\geq0$, $m\geq0$. 
\end{assumption}
(We have repeated the inclusion \eqref{eq:0a} here for the ease of further reference.) 
Let $\Pi^{(1)}_\eps(\lambda)$, $\Pi^{(2)}_\eps(\lambda)$ be as in \eqref{eq:0f}, \eqref{eq:0g}.
Next, we denote
$$
F_0(\lambda) := G \1_{(-\infty,\lambda)}(H_0) G^* \,,
\qquad
F(\lambda) :=V_0 G \1_{(-\infty,\lambda)}(H) G^* V_0\,.
$$

In order to proceed, we need a simple intermediate result.
\begin{lemma}\label{lem1}
Suppose that Assumption~\ref{ass0} holds true. Then 
for all $\lambda\in \R$ and all $\eps>0$ we have
$$
F_0(\lambda)\in \mathbf S_{p}, 
\quad F(\lambda) \in \mathbf S_{p},
\quad \Pi^{(1)}_\eps(\lambda)\in \mathbf S_{p}, 
\quad \Pi^{(2)}_\eps(\lambda)\in \mathbf S_{p}.
$$
\end{lemma}
The proof will be given in Section~\ref{sec:1a}. 

We fix some reference point $\lambda=\lambda_*\in\R$; 
our main result below concerns the 
spectral asymptotics of the operators $\Pi_\eps^{(j)}(\lambda_*)$, $j=1,2$. 
Thus, our local assumption pertains to a neighbourhood of the point $\lambda_*$:
%%%%%%%%%%%%%%%%%%
\begin{assumption}\label{ass1}
%%%%%%%%%%%%%%%%%%
There is a $\delta>0$ such that the derivatives
$$
F_0'(\lambda) = \frac{d}{d\lambda} F_0(\lambda) \,,
\qquad
F'(\lambda) = \frac{d}{d\lambda} F(\lambda)
$$
exist in the $\mathbf S_p$ norm for all $\lambda$ in the interval $[\lambda_*-\delta,\lambda_*+\delta]$ and are
H\"older continuous on this interval with some positive exponent $\varkappa>0$. 
\end{assumption} 
By a version of Privalov's theorem, Assumption~\ref{ass1} implies that the operators
\begin{equation}
T_0(z)=G(H_0-z)^{-1}G^*, \quad
T(z)=V_0G(H-z)^{-1}G^*V_0, 
\quad \im z>0,
\label{eq:1a}
\end{equation}
have limits $T_0(\lambda+i0)$, $T(\lambda+i0)$ in $\mathbf S_p$ norm for $\lambda$ in the open interval 
$(\lambda_*-\delta,\lambda_*+\delta)$, and these limits are H\"older continuous in $\lambda$ on 
this interval. In other words, Assumption~\ref{ass1} implies
a local  version of the limiting absorption principle. 
Thus, by standard results of abstract scattering theory (see, e.g., \cite[Chapter 4]{Ya1}), 
the (local) wave operators for $H_0$ and $H$ on the interval $(\lambda_*-\delta,\lambda_*+\delta)$ 
exist and the corresponding scattering matrix 
$S(\lambda)$ is well defined for $\lambda$ in this interval.

\begin{remark*}
In fact, 
we will only use the H\"older continuity of $F_0'(\lambda)$, $F'(\lambda)$ at the point $\lambda=\lambda_*$:
\begin{align*}
\|F_0'(\lambda) - F_0'(\lambda_*)\|_p &= O(|\lambda-\lambda_*|^\varkappa)
\qquad\text{as}\ \lambda\to \lambda_*  \,,
\\
\|F'(\lambda) - F'(\lambda_*)\|_p &= O(|\lambda-\lambda_*|^\varkappa)
\qquad\text{as}\ \lambda\to \lambda_* \,.
\end{align*}
The H\"older continuity as stated in Assumption~\ref{ass1} is needed only to ensure that the 
scattering matrix is well defined. 
\end{remark*}

\subsection{Main result}
As in Section~\ref{sec:0}, we denote by 
$\{ e^{i\theta_\ell(\lambda)}\}_{\ell=1}^L$, $L\leq\infty$,
the eigenvalues of $S(\lambda)$, enumerated with multiplicities taken into account,
and we use the notation $a_\ell(\lambda)$, see \eqref{eq:4}. 
Our main result is

\begin{theorem}\label{thm1}
Let Assumptions~\ref{ass0} and \ref{ass1} hold true. Let $f(t) = t^p g(t)$ with $g$ continuous on $[0,1]$. Then for $j=1,2$, one has
\begin{equation}
\label{eq:5}
\lim_{\eps\to +0} |\ln\eps|^{-1} \tr f(\Pi^{(j)}_\eps(\lambda_*)) 
= 
\frac{1}{2\pi} \sum_{\ell=1}^L 
\int_{-\infty}^\infty f \biggl( \frac{a_\ell^2(\lambda_*)}{\cosh^2(\pi x)} \biggr) \,dx
=
\int_0^1 f(t)\mu_{\lambda_*}(t)\,dt\,,
\end{equation}
where $\mu_{\lambda_*}$ is given by \eqref{eq:0ia} with $\lambda=\lambda_*$.
\end{theorem}

\noindent\textit{Discussion.}
\begin{enumerate}[(1)]
\item 
As we shall see, Assumption \ref{ass1} ensures that $\sum_{\ell=1}^L a_\ell^p(\lambda_*)<\infty$
(see  Lemma~\ref{lem9}),
and so the series in \eqref{eq:5} converges for the functions $f$ as in the hypothesis of the theorem.
\item
For $f(t)=t^n$, $n\geq p$, $n\in\mathbb N$, we obtain
\begin{align}
\label{eq:5a}
\lim_{\eps\to +0} |\ln\eps|^{-1} \tr \left(\Pi^{(j)}_\eps(\lambda_*)\right)^n 
& = 
\sum_{\ell=1}^L a_\ell^{2n}(\lambda_*)
\frac{1}{2\pi}  \int_{-\infty}^\infty (\cosh^2(\pi x))^{-n} dx
\notag \\
& =\Tr \bigl(\tfrac12\abs{S(\lambda_*)-I}\bigr)^{2n}\
\frac{1}{2\pi}  \int_{-\infty}^\infty (\cosh^2(\pi x))^{-n} dx \,.
\end{align}
Using the change of variables $y=\cosh^2(\pi x)$,
the integral in \eqref{eq:5a} can be explicitly computed,
\begin{align}
\frac{1}{2\pi} \int_{-\infty}^\infty (\cosh^2(\pi x))^{-n} dx 
& =
\frac{1}{2\pi^2}\int_1^\infty y^{-n-\frac12}(y-1)^{-\frac12}dy
\notag \\
& = 
\frac1{2\pi^2} B(n,\tfrac 12)
=
\frac{n}{2\pi^2}\frac{((n-1)!)^2}{(2n)!}2^{2n},
\label{eq:5b}
\end{align}
where $B(\cdot,\cdot)$ is the Beta function.
\item
Of course, our conditions on $f$ are far from optimal. 
For example, by a standard application of monotone convergence, 
the function $f$ in \eqref{eq:5} can be replaced by the characteristic function of any interval 
$(\alpha,\beta)$, where $0<\alpha<\beta\leq1$. 
Then Theorem~\ref{thm1} can be interpreted as the convergence of the 
eigenvalue density of $\Pi_\eps^{(j)}(\lambda_*)$ to the limiting density $\mu_{\lambda_*}$
given by the right side of \eqref{eq:0ia}. 
\item
Since our assumptions are symmetric in $H$ and $H_0$ 
(in fact, \eqref{eq:0} implies the same with $H$ in place of $H_0$, see Lemma \ref{lem4a}), 
one can see that 
the statement, identical to Theorem~\ref{thm1},  
holds true for the spectral density of the operators
$$
\1_{(-\infty,\lambda_*-\eps)}(H)
\1_{(\lambda_*+\eps,\infty)}(H_0)
\1_{(-\infty,\lambda_*-\eps)}(H)
$$
and 
$$
\psi_\eps^{-}(H-\lambda_*)
\psi_\eps^{+}(H_0-\lambda_*)
\psi_\eps^{-}(H-\lambda_*) \,.
$$
Since these operators are equivalent in the sense of the notion $\approx$ introduced in Subsection \ref{sec:1.3} to the operators
$$
\1_{(\lambda_*+\eps,\infty)}(H_0)
\1_{(-\infty,\lambda_*-\eps)}(H)
\1_{(\lambda_*+\eps,\infty)}(H_0)
$$
and
$$
\tilde\psi_\eps^{+}(H_0-\lambda_*)
\tilde\psi_\eps^{-}(H-\lambda_*)
\tilde\psi_\eps^{+}(H_0-\lambda_*) \,,
$$
(with $\tilde\psi_\eps^{+}= (\psi_\eps^{+})^{1/2}$ and $\tilde\psi_\eps^{-}= (\psi_\eps^{-})^{2}$ which again are of the form \eqref{eq:0h}), we can see that 
the statement, identical to Theorem~\ref{thm1},  
holds true also for the spectral density of the latter operators.

\item
If $m=0$ in \eqref{eq:0}, then Theorem~\ref{thm1} holds true with $f(t)=t^{p/2}g(t)$
instead of $f(t)=t^{p}g(t)$. 
Indeed, in this case one can prove Lemma~\ref{lem5} with $p$ instead of $2p$
and therefore Lemma~\ref{lem4} holds with $p/2$ instead of $p$.

\end{enumerate}

\begin{corollary}\label{cor2}
Let Assumptions~\ref{ass0} and \ref{ass1} hold true with $p=1$. Then for $j=1,2$ we have the upper bound
\begin{equation}
\label{eq:6a}
\limsup_{\eps\to 0+} |\ln\eps|^{-1} \ln\det\left(I-\Pi_\eps^{(j)}(\lambda_*)\right) 
\leq 
\frac{1}{2\pi} \sum_{\ell=1}^L \int_{-\infty}^\infty \ln \left( 1- \frac{a_\ell^2(\lambda_*)}{\cosh^2(\pi x)} \right) dx.
\end{equation}
\end{corollary}
The proof is given in Section~\ref{sec:6}.

\begin{remark*}
\begin{enumerate}[(1)]
\item
The integral on the right side of \eqref{eq:6a} can be computed and one obtains
\begin{align*}
\frac{1}{2\pi} \sum_{\ell=1}^L \int_{-\infty}^\infty \ln \left( 1- \frac{a_\ell(\lambda_*)^2}{\cosh^2(\pi x)} \right) dx
& = 
- \frac{1}{\pi^2} \sum_{\ell=1}^L \arcsin^2 a_\ell (\lambda_*)
\\
& = 
- \frac1{\pi^2} \tr \arcsin^2 \frac{|S(\lambda_*)-I|}{2}  \,.
\end{align*}
One way to see the first equality is to expand $\ln(1-x) = - \sum_{n=1}^\infty x^n/n$ and to integrate 
term by term using \eqref{eq:5b}.
The claimed formula then follows from the expansion \cite[Eq. 1.645 2]{GR}
$$
\sum_{n=1}^\infty 2^{2n-1} \frac{\left( \left(n-1\right)!\right)^2}{(2n)!} a^{2n} = \arcsin^2 a \,,
\qquad |a|\leq 1 \,.
$$
This computation is very similar to a computation in \cite{GKMO}.
\item
Under our assumptions, it is not possible to obtain any lower bound in \eqref{eq:6a}. 
Indeed, a single eigenvalue $=1$ of $\Pi_\eps^{(j)}(\lambda_*)$ can make 
the determinant vanish. Such examples are easy to construct in the abstract setting 
discussed here. 
\end{enumerate}
\end{remark*}

\subsection{Application to the Schr\"odinger operator}

Let $H_0=-\Delta$ in $L^2(\R^d)$, $d\geq1$, and let the real-valued potential $V=V(x)$, $x\in\R^d$, 
satisfy
\begin{equation}
\abs{V(x)}\leq C(1+\abs{x})^{-\rho}, \quad \rho>1.
\label{eq:6b}
\end{equation}
Denote $H=H_0+V$. 

\begin{lemma}
Assume \eqref{eq:6b}.
Then: 
\begin{enumerate}[\rm (i)]
\item
Assumption~\ref{ass0} is satisfied with any $p\geq1$, $p>d/\rho$;
\item
Assumption~\ref{ass1} is satisfied with any $p>(d-1)/(\rho-1)$. 
\end{enumerate}
\end{lemma}

Thus, the bound \eqref{eq:6b} ensures that Theorem~\ref{thm1} applies with any 
$$
p\geq 1\,,\qquad p>\max\{d/\rho,(d-1)/(\rho-1)\}. 
$$
In particular, if $\rho>d$, then Corollary~\ref{cor2} applies.

\begin{proof}
(i) For $(1+2m)2p>d$ we shall verify the inclusion
$$
\abs{V}^{1/2}(-\Delta+I)^{-\frac12-m}\in \mathbf S_{2p}.
$$
Under our restrictions on $p$, $\rho$, $m$ we have
$$
\int_{\R^d} (\abs{V(x)}^{1/2})^{2p}\,dx<\infty
\qquad \text{ and } \qquad
\int_{\R^d}((\abs{\xi}^2+1)^{-\frac12-m})^{2p}\,d\xi<\infty \,,
$$
and therefore the above inclusion follows from the Kato--Seiler--Simon bound \cite[Thm. 4.1]{S1}.

(ii)
By the standard (operator norm) limiting absorption principle, the derivatives
$F_0'(\lambda)$, $F'(\lambda)$ exist for all $\lambda>0$ and are given by
\begin{equation}
F_0'(\lambda)=\frac1\pi \im T_0(\lambda+i0) \,, 
\qquad
F'(\lambda)=\frac1\pi \im T(\lambda+i0) \,, 
\label{eq:6c}
\end{equation}
where $T_0$, $T$ are defined in \eqref{eq:1a}. 
The inclusion $F_0'(\lambda)\in \mathbf S_p$, $p>(d-1)/(\rho-1)$ and
the H\"older continuity of this operator in $\mathbf S_p$ norm are well 
known facts; see, e.g., \cite[Lemma 8.1.8]{Ya2}.
The corresponding statements concerning $F'(\lambda)$ follow by an 
application of the resolvent identity. More specifically, one of the versions
of the resolvent identity can be written as
\begin{equation}
T(z)=V_0-(I+V_0T_0(z))^{-1}V_0.
\label{eq:6d}
\end{equation}
Taking the imaginary part here and using \eqref{eq:6d} again, we obtain
\begin{multline*}
\im T(z)
=
(I+V_0T_0(z))^{-1}V_0(\im T_0(z))V_0(I+T_0(z)^*V_0)^{-1}
\\
=
(V_0-T(z))(\im T_0(z))(V_0-T(z)^*).
\end{multline*}
Passing to the limit $z\to\lambda+i0$ and using \eqref{eq:6c}, we arrive at 
the identity
$$
F'(\lambda)=(V_0-T(\lambda+i0)) F_0'(\lambda) (V_0-T(\lambda+i0)^*). 
$$
This yields the required statements for $F'(\lambda)$. 
\end{proof}

\subsection{Key ideas in the proof of Theorem~\ref{thm1}}\label{sec:1.4}

The main task is to prove Theorem~\ref{thm1} for the operator $\Pi_\eps^{(1)}(\lambda_*)$; the 
statement for $\Pi_\eps^{(2)}(\lambda_*)$ easily follows by some monotonicity arguments. 
In what follows, for simplicity of notation we take $\lambda_*=0$
and set $\Pi^{(j)}_\eps:=\Pi^{(j)}_\eps(0)$, $a_\ell:=a_\ell(0)$. 
Our first step is a spectral localisation lemma (Lemma \ref{lem4}): we show that 
the operator $\Pi^{(1)}_\eps$,
\begin{equation}
\Pi^{(1)}_\eps 
= 
\1_{(-\infty,-\eps)}(H_0) 
\1_{(\eps,\infty)} (H)
\1_{(-\infty,-\eps)}(H_0) \,,
\label{eq:16a}
\end{equation} 
can be replaced by the operator
\begin{equation}
\label{eq:16}
\wt \Pi^{(1)}_\eps
= 
\1_{(-\delta,-\eps)}(H_0) 
\1_{(\eps,\delta)} (H)
\1_{(-\delta,-\eps)}(H_0)\,,
\end{equation}
where $\delta$ is defined in Assumption~\ref{ass1}. 
This is a standard argument using resolvent identities and some functional calculus for self-adjoint operators. 
Next, the key step is the product representation (Lemma~\ref{factorization})
\begin{equation}
\1_{(\eps,\delta)}(H)\1_{(-\delta,-\eps)}(H_0) 
= 
\mathcal Z_\eps  \left(\mathcal Z_\eps^{(0)}\right)^*.
\label{eq:22}
\end{equation}
Here, the operators 
$\mathcal Z_\eps, \mathcal Z^{(0)}_\eps: L^2(\R_+,\mathcal K)\to\mathcal H$  are defined by
\begin{align}
\mathcal Z_\eps^{(0)}f & = \int_0^\infty e^{t H_0} \1_{(-\delta,-\eps)}(H_0)G^*f(t)\,dt,
\label{eq:22a}
\\
\mathcal Z_\eps f & = \int_0^\infty e^{-t H} \1_{(\eps,\delta)}(H) G^*V_0f(t)\,dt.
\label{eq:22b}
\end{align}
Of course, from  \eqref{eq:22} it follows that 
the operator $\wt \Pi_\eps^{(1)}$ 
admits the factorization
$$
\wt \Pi_\eps^{(1)}
= 
\mathcal Z_\eps^{(0)} \left(\mathcal Z_\eps\right)^* \mathcal Z_\eps  \left(\mathcal Z_\eps^{(0)}\right)^*.
$$
Further, it turns out that the products $K_\eps={(\mathcal Z_\eps)}^* \mathcal Z_\eps$ 
and $K_\eps^{(0)}={(\mathcal Z_\eps^{(0)})}^* \mathcal Z_\eps^{(0)}$ 
are integral Hankel type operators in 
$L^2(\R_+,\mathcal K)$. That is, these operators have the form 
\begin{equation}
K_\eps:
f\mapsto \int_0^\infty k_\eps(t+s)f(s) ds, \quad t>0,
\label{eq:23c}
\end{equation}
where $k_\eps=k_\eps(t)$ is some operator valued function, 
called the \emph{kernel of $K_\eps$}.
As $\eps\to0$, the operators  $K_\eps$, $K_\eps^{(0)}$ can be approximated 
in $\mathbf S_p$ norm by some ``model Hankel operators'' with explicit integral kernels. 
By using the $\approx$ relation (see Section~\ref{sec:1.3}), this allows us to 
reduce the problem to computing traces of powers of these model Hankel operators (Lemma \ref{lem2}). 
The latter turns out to be a relatively easy task. 

\begin{remark}
In \cite{GKMO}, the authors consider the traces $\tr (\Pi_\eps(\lambda))^n$ for the regularisation
$\Pi_\eps(\lambda)$ similar to our $\Pi_\eps^{(2)}(\lambda)$. Through a series of transformations, 
the computation of the leading term of the asymptotics of this trace is reduced to the evaluation 
of some explicit multiple ($n$-fold) integral. 
It is curious that Hankel operators do appear in \cite{GKMO}, but only in passing, as a 
tool for evaluation of this integral. One of the points of this work is to emphasize that
Hankel operators are at the heart of the matter here. 
\end{remark}

\begin{remark}
Much of the technique of the paper is borrowed from \cite{P1,PYa1}. 
Crucially, the idea of the factorization \eqref{eq:22} and the analysis of the operators $K_\eps$ and $K_\eps^{(0)}$ 
comes from \cite{P1}.
\end{remark}

\subsection{The structure of the paper}
In Section~\ref{sec:1a} we prove the preliminary Lemma~\ref{lem1}.
In Sections~\ref{sec:2} and \ref{sec:3} we prepare some auxiliary statements concerning Hankel
operators. More precisely, in Section~\ref{sec:2} we compute the asymptotics 
of traces of powers of a model Hankel operator and in Section~\ref{sec:3} we 
present some $\mathbf S_p$ class estimates for operator valued integral Hankel operators. 
In Section~\ref{sec:5} we analyse the operators $\mathcal Z_\eps$, $\mathcal Z_\eps^{(0)}$, 
see \eqref{eq:22a}, \eqref{eq:22b}. 
In Section~\ref{sec:4} we prove the spectral localization lemma, which reduces the 
analysis of $\Pi_\eps^{(1)}$ to that of $\wt \Pi_\eps^{(1)}$, see \eqref{eq:16a}, \eqref{eq:16}. 
In Section~\ref{sec:6} we prove the main results of the paper.

%%%%%%%%%%%%%%%%%%%%%%%%%%%%%%%%%%%
\section{Proof of Lemma \ref{lem1}}\label{sec:1a}
%%%%%%%%%%%%%%%%%%%%%%%%%%%%%%%%%%%

Here we prepare some auxiliary statements which will be required in the 
proof of the spectral localization lemma ($=$ Lemma~\ref{lem4}) 
and prove Lemma~\ref{lem1}.

\begin{lemma}\label{lem4a}
Let Assumption~\ref{ass0} hold true. Then 
$$
G(H+M)^{-\frac12-m}\in \mathbf S_{2p}
$$
for all $M>-\inf\sigma(H)$ and the same exponents $m$, $p$ as in \eqref{eq:0}.
\end{lemma}
\begin{proof}
This is a straightforward adaptation of the argument of the proof of \cite[Theorem XI.12]{RS3}, 
where a variant of the above statement was proven for $p=1/2$. For completeness, below we outline the proof. 
Choose $M>-\min\{\inf\sigma(H_0),\inf\sigma(H)\}$ sufficiently large so that 
\begin{equation}
\norm{(H_0+M)^{-\frac12}V(H_0+M)^{-\frac12}}\leq r<1.
\label{eq:17d}
\end{equation}
In order to make our formulas below more readable, set $h_0=H_0+M$ and $h=H+M$. 
Write
$$
Gh^{-m-\frac12}=Gh^{-m}h_0^{-\frac12}h_0^{\frac12}h^{-\frac12}.
$$
Since the operators $h_0$ and $h$ have the same form domain, the product $h_0^{\frac12}h^{-\frac12}$ is bounded. 
We see that it suffices to prove the inclusion
\begin{equation}
Gh^{-m}h_0^{-\frac12}\in \mathbf S_{2p}.
\label{eq:17b}
\end{equation}
We have
$$
h^{-1}
=
h_0^{-\frac12}\left\{\sum_{j=0}^\infty (h_0^{-\frac12}(-V)h_0^{-\frac12})^j\right\}h_0^{-\frac12},
$$
where by \eqref{eq:17d} the series converges in the operator norm. It follows that
\begin{equation}
Gh^{-m}h_0^{-\frac12}
=
Gh_0^{-\frac12}
\sum_k \prod_{i=1}^k(h_0^{-\frac12}(-V)h_0^{-\frac12-\ell_i}),
\label{eq:17c}
\end{equation}
where the sum is taken over the set of terms with $\ell_1+\dots+\ell_k=m$. 
By interpolation between the two inclusions in \eqref{eq:0} we obtain
$$
Gh_0^{-\frac12-\ell}\in\mathbf S_{2pm/\ell}, \quad 0<\ell\leq m,
$$
and therefore, using the H\"older inequality for Schatten classes, 
we see that each term in \eqref{eq:17c} satisfies
$$
\prod_{i=1}^k(h_0^{-\frac12}(-V)h_0^{-\frac12-\ell_i}) \in \mathbf S_{2p}, 
\quad
\ell_1+\dots+\ell_k=m.
$$
Moreover, as in \cite[Theorem XI.12]{RS3}, using condition \eqref{eq:17d}, we obtain 
the estimate
$$
\Norm{\prod_{i=1}^k(h_0^{-\frac12}(-V)h_0^{-\frac12-\ell_i})}_{2p}\leq Cr^k
$$
for each term in the series \eqref{eq:17c} over $k$. 
It follows that 
the series in \eqref{eq:17c} converges absolutely in the norm of $\mathbf S_{2p}$.
Thus, we obtain \eqref{eq:17b}. 
\end{proof}

\begin{lemma}\label{lem5}
Let Assumption~\ref{ass0} hold true.  Then for any $\lambda_1<\lambda_2$, we have
\begin{align}
\1_{(-\infty,\lambda_1)}(H_0)
\1_{(\lambda_2,\infty)}(H)
&\in \mathbf S_{2p},
\label{eq:18a}
\\
\1_{(-\infty,\lambda_1)}(H)
\1_{(\lambda_2,\infty)}(H_0)
&\in \mathbf S_{2p}.
\label{eq:18b}
\end{align}
\end{lemma}

\begin{proof}
First note that by Assumption~\ref{ass0} and by  Lemma~\ref{lem4a}, we have
\begin{equation}
G\1_{(-\infty,\lambda)}(H_0)\in \mathbf S_{2p}, 
\quad
G\1_{(-\infty,\lambda)}(H)\in \mathbf S_{2p}, 
\quad
\forall \lambda\in\R.
\label{eq:18c}
\end{equation}
Next, let $\lambda_{\min}<\min\{\inf\sigma(H_0), \inf\sigma(H)\}$. 
Choose a function $\psi\in C_0^\infty(\R)$ such that
$$
\psi(x)
=
\begin{cases}
0 & \text{ for $x\leq \lambda_{\min}-1$ and for $x\geq\lambda_2$,}
\\
1 & \text{ for $\lambda_{\min}\leq x\leq \lambda_1$.}
\end{cases}
$$
Then 
$$
\1_{(-\infty,\lambda_1)}(H_0)\psi(H_0)
=
\1_{(-\infty,\lambda_1)}(H_0)
\quad\text{ and }\quad
\psi(H) \1_{(\lambda_2,\infty)}(H)=0\, ,
$$
and therefore
\begin{equation}
\1_{(-\infty,\lambda_1)}(H_0)\1_{(\lambda_2,\infty)}(H)
=
\1_{(-\infty,\lambda_1)}(H_0)(\psi(H_0)-\psi(H))\1_{(\lambda_2,\infty)}(H)\, .
\label{eq:18f}
\end{equation}
Since $\psi\in C_0^\infty(\R)$, we can use a standard method based on 
almost analytic continuation of $\psi$ (see e.g. \cite[Chapter 8]{DS})
to represent it as
$$
\psi(\lambda)=\int_\C\frac{\nu(z)}{\lambda-z}dL(z),
$$
where $dL(z)$ is the 2-dimensional Lebesque measure in $\C$, and $\nu$ is some 
function, compactly supported in $\C$ and satisfying the estimate
\begin{equation}
\nu(z)=O(\abs{\im z}^N), \quad \im z\to0, \quad \forall N>0.
\label{eq:18d}
\end{equation}
Then, by the resolvent identity, 
\begin{equation}
\1_{(-\infty,\lambda_1)}(H_0)(\psi(H_0)-\psi(H))
=
\int_\C \1_{(-\infty,\lambda_1)}(H_0)(H_0-z)^{-1}V(H-z)^{-1}\nu(z)dL(z).
\label{eq:18e}
\end{equation}
Let us prove that the above integral converges absolutely in $\mathbf S_{2p}$. 
We have
\begin{multline*}
\1_{(-\infty,\lambda_1)}(H_0)(H_0-z)^{-1}V(H-z)^{-1}
\\
=
(H_0-z)^{-1} (G\1_{(-\infty,\lambda_1)}(H_0))^*V_0 (G(H-i)^{-1})(H-i)(H-z)^{-1}, 
\end{multline*}
and therefore
\begin{multline*}
\norm{\1_{(-\infty,\lambda_1)}(H_0)(H_0-z)^{-1}V(H-z)^{-1}}_{2p}
\\
\leq
\norm{(H_0-z)^{-1}}\norm{G\1_{(-\infty,\lambda_1)}(H_0)}_{2p}
\norm{V_0}\norm{G(H-i)^{-1}}\norm{(H-i)(H-z)^{-1}}
\\
=
O(\abs{\im z}^{-2}), \quad \im z\to0.
\end{multline*}
Combining this with \eqref{eq:18d}, we obtain that the integral in \eqref{eq:18e}
converges absolutely in the norm of $\mathbf S_{2p}$. 
In view of \eqref{eq:18f}, this yields the inclusion \eqref{eq:18a}. 
The second inclusion \eqref{eq:18b} is proven by following exactly the same sequence
of steps. 
\end{proof}

\begin{proof}[Proof of Lemma \ref{lem1}]
We have
$$
F_0(\lambda)=(G\1_{(-\infty,\lambda)}(H_0))(G\1_{(-\infty,\lambda)}(H_0))^*,
$$
where, by  the first inclusion in \eqref{eq:18c}, both factors are in $\mathbf S_{2p}$. 
Using the H\"older inequality for Schatten classes, we obtain that $F_0(\lambda)\in \mathbf S_{p}$. 
Similarly, using the second inclusion  in \eqref{eq:18c}, we obtain $F(\lambda)\in \mathbf S_{p}$. 
Further, $\Pi^{(1)}_\eps$ can be written as
$$
\Pi^{(1)}_\eps
=
\bigl(\1_{(-\infty,-\eps)}(H_0)\1_{(\eps,\infty)}(H)\bigr)
\bigl(\1_{(\eps,\infty)}(H)\1_{(-\infty,-\eps)}(H_0)\bigr)\,,
$$
where both terms are in $\mathbf S_{2p}$ by Lemma~\ref{lem5}.
Thus, $\Pi_\eps^{(1)}\in \mathbf S_p$. 
Finally, consider $\Pi^{(2)}_\eps$; by the definition of the functions $\psi_\eps^{\pm}$, we have
\begin{align*}
\Pi^{(2)}_\eps
& =
\psi_\eps^{-}(H_0)
\psi_\eps^{+}(H)
\psi_\eps^{-}(H_0)
\\
& =
\psi_\eps^{-}(H_0)
\1_{(-\infty,-\eps)}(H_0)\1_{(\eps,\infty)}(H)
\psi_\eps^{+}(H)
\1_{(\eps,\infty)}(H)\1_{(-\infty,-\eps)}(H_0)
\psi_\eps^{-}(H_0)\,,
\end{align*}
and so the result again follows by  Lemma~\ref{lem5}.
\end{proof}

%%%%%%%%%%%%%%%%%%%%%%%%%%%%%%%%%%%

%%%%%%%%%%%%%%%%%%%%%%%%%%%%%%%%%%%
\section{Spectral density of a Hankel operator}\label{sec:2}
%%%%%%%%%%%%%%%%%%%%%%%%%%%%%%%%%%%

For $0<\eps\leq\delta$ let $\Gamma_\eps$ be the Hankel-type integral operator 
in $L^2(\R_+)$ with the integral kernel $\gamma_\eps(s+t)$, $s,t\in\R_+$, 
where $\gamma_\eps$ is given by
$$
\gamma_\eps(t)
=
\int_\eps^\delta e^{-t\lambda}\,d\lambda = \frac{e^{-t\eps}-e^{-t\delta}}{t} \,,  
\qquad t\in\R_+ \,.
$$

\begin{lemma}\label{lem2}
The operator $\Gamma_\eps$ belongs to the trace class and satisfies
\begin{equation}
\label{eq:6}
\Gamma_\eps\geq 0 \,,
\qquad
\|\Gamma_\eps\|\leq \pi \,.
\end{equation}
Moreover, for any $q\geq 1$, one has
\begin{equation}
\label{eq:7}
\lim_{\eps\to+0} |\ln\eps|^{-1} \tr \left(\Gamma_\eps\right)^q 
= 
\frac{1}{2\pi} \int_{-\infty}^\infty \left( \frac{\pi}{\cosh(\pi x)}\right)^q dx \,.
\end{equation}
\end{lemma}

\noindent\textit{Remark.}
In fact, one can check that $\Gamma_\eps\in\mathbf S_q$ for all $q>0$ and that \eqref{eq:7} holds for all $q>0$.

\medskip

The proof of Lemma \ref{lem2} relies on some well-known facts about the Carleman operator, that is, 
the Hankel operator in $L^2(\R_+)$ with the integral kernel $(t+s)^{-1}$, which we recall next. 
Let $\mathcal L$ be the (self-adjoint) operator of the Laplace transform in $L^2(\R_+)$:
\begin{equation}
(\mathcal L f)(t)=\int_0^\infty e^{-xt}f(x)dx, \quad
f\in L^2(\R_+).
\label{eq:laplace}
\end{equation}
Clearly, the Carleman operator can be written as $\mathcal L^2$. 
This operator can be explicitly diagonalized. Namely, let $U:L^2(\R_+)\to L^2(\R)$ be the unitary operator defined by
$$
(Uf)(x) = e^{x/2} f(e^x) \,,
\qquad x\in\R \,,\ f\in L^2(\R_+) \,.
$$
By an explicit calculation
(see e.g. \cite[Section 10.2]{Peller}), we obtain that 
$U \mathcal L^2 U^*$ is the operator of convolution 
with the function $1/(2\cosh(x/2))$. Computing the Fourier transform of this function,
\begin{equation}
\int_{-\infty}^\infty \frac{e^{-i\xi x}}{2 \cosh(x/2)}dx = \frac{\pi}{\cosh(\pi\xi)}=:b(\xi),
\label{eq:7b}
\end{equation}
we obtain
\begin{equation}
\label{eq:7a}
U \mathcal L^2 U^* = b(D).
\end{equation}
Here, the operator $D$ (as well as the operator $X$, needed later) 
are the self-adjoint operators  given by
\begin{equation}
(Xf)(x) = xf(x) \,,
\qquad
(Df)(x) = -i \frac{d}{dx} f(x) 
\qquad 
\text{ in $L^2(\R)$.}
\label{eq:7c}
\end{equation}
Since $\|b\|_{L^\infty} =\pi$, we note that \eqref{eq:7a} implies, in particular, that
\begin{equation}
\label{eq:8a}
\| \mathcal L \| = \sqrt\pi \,.
\end{equation}

\begin{proof}[Proof of Lemma~\ref{lem2}]
Along with the operator $X$ defined by \eqref{eq:7c}, we will need its half-line version
$$
(X_+f)(x) = xf(x) \,,
\qquad 
\text{ in $L^2(\R_+)$.}
$$
In terms of the Laplace transform $\mathcal L$, our operator $\Gamma_\eps$ can be factorized as
\begin{equation}
\label{eq:8}
\Gamma_\eps = \mathcal L \mathbbm 1_{(\eps,\delta)}(X_+) \mathcal L 
= 
\left( \mathbbm 1_{(\eps,\delta)}(X_+) \mathcal L  \right)^* \left( \mathbbm 1_{(\eps,\delta)}(X_+) \mathcal L  \right) \,.
\end{equation}
This proves that $\Gamma_\eps\geq 0$ and, since it is easy to check that $\mathbbm 1_{(\eps,\delta)}(X_+) \mathcal L \in\mathbf S_2$, it follows that $\Gamma_\eps$ is trace class. Moreover, \eqref{eq:8a} implies $\|\Gamma_\eps\|\leq\pi$.

Let us prove \eqref{eq:7}. 
Using the notation $\approx$ introduced in Section~\ref{sec:1.3}, 
we deduce from \eqref{eq:8} that $\Gamma_\eps \approx \mathbbm 1_{(\eps,\delta)}(X_+) \mathcal L^2 \mathbbm 1_{(\eps,\delta)}(X_+)$. 
Thus, it follows from \eqref{eq:7a} that
$$
\Gamma_\eps \approx 
\1_{(\eps,\delta)}(X_+) U^* b(D) U \1_{(\eps,\delta)}(X_+) 
= U^* \1_{(\ln\eps,\ln\delta)}(X) b(D) \1_{(\ln\eps,\ln\delta)}(X) U.
$$
From here we obtain 
$$
\tr f\left(\Gamma_\eps\right)
= 
\tr f\left( \mathbbm 1_{(\ln\eps,\ln\delta)}(X) b(D) \mathbbm 1_{(\ln\eps,\ln\delta)}(X) \right) \,, 
$$
for any continuous function with $f(0)=0$.
 
Now we first observe that for $q=1$, formula \eqref{eq:7} is a direct calculation 
of the trace of $\mathbbm 1_{(\ln\eps,\ln\delta)}(X) b(D)$. 
For $q\geq 2$, we employ the following result of \cite{LaSa1}. 
Let $P$ be an orthogonal projection in a Hilbert space and let $B$ be a self-adjoint operator 
such that $P B$ is Hilbert--Schmidt. Then for any $f\in C^2(\R)$ with $f(0)=0$, one has
\begin{equation}
\label{eq:10}
\left| \tr f(P BP) - \tr P f(B) P \right| \leq \frac12 \|f''\|_{L^\infty} \|P B(1-P)\|_2^2 \,.
\end{equation}
Let us take
$P = \mathbbm 1_{(\ln\eps,\ln\delta)}(X)$, $B= b(D)$, and $f(t)= t^q$, $q\geq 2$. 
Then 
\begin{equation}
\tr f(P BP)=\tr (\Gamma_\eps)^q\,,
\label{eq:10a}
\end{equation}
and
\begin{align}
\tr P f(B)P & = \frac{1}{2\pi} \int_\R \mathbbm 1_{(\ln\eps,\ln\delta)}(x)\,dx \int_{\R}b(\xi)^q \,d\xi 
\notag \\
& = 
\left( |\ln\eps|+ O(1)\right) \frac{1}{2\pi} \int_{-\infty}^\infty \left( \frac{\pi}{\cosh(\pi x)}\right)^q dx \,.
\label{eq:10b}
\end{align}
Thus, it remains to estimate the right side in \eqref{eq:10}:
$$
\|P B (1-P) \|_2^2 = \|(P B - B P)(1-P)\|_2^2 \leq \| [P,B]\|_2^2 \,.
$$
Formula \eqref{eq:7b} implies that $[P,B]$ has integral kernel
$$
\frac{\mathbbm 1_{(\ln\eps,\ln\delta)}(x) - \mathbbm 1_{(\ln\eps,\ln\delta)}(y)}{2\ \cosh((x-y)/2)} \,,
\qquad x,y\in\R \,.
$$
Thus,
$$
\| [P,B]\|_2^2 
= 
\iint_{\R\times\R} 
\frac{\left(\mathbbm 1_{(\ln\eps,\ln\delta)}(x) - \mathbbm 1_{(\ln\eps,\ln\delta)}(y) \right)^2}{4\ \cosh^2((x-y)/2)} 
\,dx\,dy 
= 
\int_\R \frac{\varphi(z)}{4\ \cosh^2(z/2)} \,dz
$$
with
$$
\varphi(z) = \int_\R \left(\mathbbm 1_{(\ln\eps,\ln\delta)}(y+z) - \mathbbm 1_{(\ln\eps,\ln\delta)}(y) \right)^2 \,dy 
= 
2\min\{|z|,\ln\delta-\ln\eps\} \leq 2|z|.
$$
We obtain $\| [P,B]\|_2^2 \leq \int_\R \frac{|z|}{2\ \cosh^2(z/2)} \,dz<\infty$ uniformly in $\eps>0$.
\iffalse
The operator $P$ acts as convolution with the function
$$
g_\eps(x) = \frac 1{2\pi} \int_{\ln\eps}^{\ln\delta} e^{i\xi x}\,d\xi = \frac{e^{ix\ln\delta} - e^{ix\ln\eps}}{2\pi x} \,,
$$
so
$$
|g_\eps(x)| \leq \frac{1}{\pi |x|} \,.
$$
Thus,
\begin{align}
\label{eq:11}
\| [P,b(X)^2]\|_2^2 & = \iint_{\R\times\R} |b(x)-b(y)|^2 |g_\eps(x-y)|^2\,dx\,dy \\
& = \iint_{\R\times\R} |g_\eps(z)|^2 |b(x-z) - b(x)|^2 \,dx\,dz \notag \\
& \leq \frac{1}{\pi^2} \int \frac{\psi(z)}{|z|^2} \,dz \,,
\end{align}
where $\psi(z)=\int_\R |b(x-z)-b(x)|^2 \,dx$ satisfies $\psi\in L^\infty$ and $\psi(z)= O(|z|^2)$ as $|z|\to 0$. Thus, \eqref{eq:11} is bounded uniformly with respect to $\eps$.
\fi
Returning to \eqref{eq:10}, we obtain
$$
\left| \tr f(P BP) - \tr P f(B)P \right| \leq C \,.
$$
(Here we also used the fact that the $L^\infty$-norm of $f''$ needs only be evaluated on the finite interval $[0,\|B\|]=[0,\pi]$.) 
Combining this with \eqref{eq:10a}, \eqref{eq:10b}, we obtain the required statement for $f(t)=t^q$, $q\geq 2$.

Now assume that $f(t)=t^q$ with $1<q<2$ or, more generally, that $f(t)=t g(t)$ with $g$ continuous on $[0,1]$. Then, for any $\delta>0$, there is a polynomial $P$ with $\|P-g\|_\infty\leq\delta$. Let $f^{(1)}(t) = t P(t)$ and $f^{(2)}(t)= t(g(t)-P(t))$. According to the first part of the proof,
$$
\lim_{\eps\to 0} |\ln\eps|^{-1} \tr f^{(1)}(\Gamma_\eps) = \frac1{2\pi} \int_{-\infty}^\infty f^{(1)}\left(\frac{\pi}{\cosh(\pi x)}\right) dx 
$$
and, by \eqref{eq:10b} with $q=1$,
$$
\left| \tr f^{(2)}(\Gamma_\eps) \right| \leq \delta \tr\Gamma_\eps = \delta (|\ln\eps|+O(1)) \frac{1}{2\pi} \int_{-\infty}^\infty \frac{\pi}{\cosh(\pi x)} dx \leq C_1 \delta |\ln\eps| \,.
$$
On the other hand,
\begin{align*}
& \left| \frac1{2\pi} \int_{-\infty}^\infty f^{(1)}\left(\frac{\pi}{\cosh(\pi x)}\right) dx
- \frac1{2\pi} \int_{-\infty}^\infty f\left(\frac{\pi}{\cosh(\pi x)}\right) dx \right| \\
& \qquad \leq \frac1{2\pi} \int_{-\infty}^\infty \left| f^{(2)}\left(\frac{\pi}{\cosh(\pi x)}\right) \right| dx \\
& \qquad \leq \frac\delta{2\pi} \int_{-\infty}^\infty \frac{\pi}{\cosh(\pi x)} dx = C_2\delta \,.
\end{align*}
Thus,
$$
\limsup_{\eps\to 0} \left| |\ln\eps|^{-1} \tr f(\Gamma_\eps) - \frac1{2\pi} \int_{-\infty}^\infty f\left(\frac{\pi}{\cosh(\pi x)}\right) dx \right| \leq (C_1+C_2) \delta \,.
$$
Since $\delta$ is arbitrary, we obtain the asymptotics for any $f$ of the above form and, in particular, for $f(t)=t^q$, $1<q<2$.
\end{proof}

\begin{remark}
\begin{enumerate}
\item
It is clear from the proof that if $q\geq2$ or $q=1$, then in fact we have a stronger statement:
$$
\tr \left(\Gamma_\eps\right)^q 
= 
|\ln\eps| \frac{1}{2\pi} \int_{-\infty}^\infty \left( \frac{\pi}{\cosh(\pi x)}\right)^q dx
+
O(1), 
\quad 
\eps\to+0.
$$
\item
The crucial fact used in the proof above that $\Gamma_\eps$ is unitarily equivalent to the pseudo-differential operator
\begin{equation*}
\mathbbm 1_{(\eps,\delta)}(e^X)b(D)\mathbbm 1_{(\eps,\delta)}(e^X)
\qquad\text{in}\ L^2(\R) \,,
\end{equation*}
is a special case of a more general result of Yafaev \cite{Ya3}; see also \cite{Widom} for an older related result. 
\item
Note that the standard Berezin--Lieb inequality (see, e.g., \cite{LaSa1}) yields for arbitrary real numbers $q\geq 1$ the one-sided bound
\begin{align*}
\tr \left(\Gamma_\eps\right)^q & \leq
\frac{1}{2\pi} \int_\R \mathbbm 1_{(\ln\eps,\ln\delta)}(x)\,dx \int_{-\infty}^\infty \left( \frac{\pi}{\cosh(\pi \xi)}\right)^q d\xi \\
& = \frac{1}{2\pi} \left( |\ln\eps|+ O(1)\right) \int_{-\infty}^\infty \left( \frac{\pi}{\cosh(\pi \xi)}\right)^q d\xi \,.
\end{align*}
Thus the approximation argument above was only needed for a lower bound for $1<q<2$.
\end{enumerate}
\end{remark}

%%%%%%%%%%%%%%%%%%%%%%%%%

\section{Estimates for operator-valued Hankel operators}\label{sec:3}

We need some Schatten class estimates for Hankel operators acting in $L^2(\R_+,\mathcal K)$. 
Fix $q\geq1$; 
let $\sigma: \R_+\to \mathbf S_q(\mathcal K)$ be a measurable function and let $K$ be the integral Hankel 
operator in $L^2(\R_+,\mathcal K)$ (see \eqref{eq:23c}) 
with the kernel $k=k(t)$, given by the Laplace transform of $\sigma$:
$$
k(t) = \int_0^\infty e^{-\lambda t} \sigma(\lambda)\,d\lambda \,.
$$

\begin{lemma}\label{lem3}
For $1\leq q<\infty$, one has
$$
\|K\|_q^q \leq \pi^{q-1} \int_0^\infty \|\sigma(\lambda)\|_q^q \,\frac{d\lambda}{2\lambda}
$$
and
$$
\|K\| \leq \pi \esssup_{\lambda>0} \|\sigma(\lambda)\| \,.
$$
\end{lemma}
\begin{remark*}
In the scalar case $\mathcal K=\C$, 
this result has appeared in \cite{Widom} for $q=\infty$ and in \cite{Howland} for $q=1$. 
\end{remark*}

\begin{proof}
We can write $K=\mathcal L \sigma \mathcal L$, where $\mathcal L$ denotes, 
as in \eqref{eq:laplace}, the Laplace transform. Then, for $q=\infty$, we have
\begin{equation}
\label{eq:13a}
\|K\| 
\leq 
\|\mathcal L\|^2 \|\sigma\| 
= 
\pi \esssup\limits_{\lambda>0}\|\sigma(\lambda)\| \,,
\end{equation}
where we have used \eqref{eq:8a}. For $q=1$, we have
\begin{equation}
\label{eq:13b}
\|K\|_1 \leq \|\mathcal L |\sigma|^{1/2}\|_2^2 
= 
\int_0^\infty \| |\sigma(\lambda)|^{1/2} \|_2^2 \,\frac{d\lambda}{2\lambda} 
= 
\int_0^\infty \|\sigma(\lambda)\|_1 \,\frac{d\lambda}{2\lambda} \,.
\end{equation}
For $1<q<\infty$, the bound follows by complex interpolation. For the sake of completeness we include the details of this argument. For fixed $1<q<\infty$ we consider the analytic family of operators
$$
K_z = \mathcal L u |\sigma|^{zq} \mathcal L \,,
$$
where, for $\lambda>0$, $u(\lambda)$ is a partial isometry in $\mathcal K$ such that $\sigma(\lambda) = u(\lambda) |\sigma(\lambda)|$. The bounds \eqref{eq:13a} and \eqref{eq:13b} show that $K_z\in\mathbf B$ ($=$ the class of bounded operators) 
if $\re z=0$ with
$$
\|K_z\| \leq \pi \,,
$$
and that $K_z\in\mathbf S_1$ if $\re z=1$ with
$$
\|K_z\|_1 \leq \int_0^\infty \| |\sigma(\lambda)|^{zq} \|_1 \,\frac{d\lambda}{2\lambda} 
= 
\int_0^\infty \|\sigma(\lambda)\|_q^q \,\frac{d\lambda}{2\lambda} \,.
$$
Thus, by complex interpolation (see, e.g., \cite[Thm. 2.9]{S1}), $K_{1/q}\in\mathbf S_q$ with
$$
\|K_{1/q}\|_q 
\leq 
\pi^{1-1/q} \left( \int_0^\infty \|\sigma(\lambda)\|_q^q\,\frac{d\lambda}{2\lambda} \right)^{1/q} \,.
$$
Since $K_{1/q}=\mathcal L \sigma\mathcal L=K$, this proves the lemma.
\end{proof}

%%%%%%%%%%%%%%%%%%%%%%%%%%%%%%%%%%%%%%%%%

%%%%%%%%%%%%%%%%%%%%%%%%%%%%%%%%%%%%%%%%%%%%%%%%%%%
\section{The operators $\mathcal Z_\eps$ and $\mathcal Z^{(0)}_\eps$}\label{sec:5}
%%%%%%%%%%%%%%%%%%%%%%%%%%%%%%%%%%%%%%%%%%%%%%%%%%%
Let
$\mathcal Z_\eps, \mathcal Z^{(0)}_\eps: L^2(\R_+,\mathcal K)\to\mathcal H$ 
be the operators defined by \eqref{eq:22a}, \eqref{eq:22b}, 
and let $K_\eps={(\mathcal Z_\eps)}^* \mathcal Z_\eps$, 
$K_\eps^{(0)}={(\mathcal Z_\eps^{(0)})}^* \mathcal Z_\eps^{(0)}$. 
An inspection shows that $K_\eps$ and $K_\eps^{(0)}$ are Hankel operators in $L^2(\R_+,\mathcal K)$
with the kernels given by 
$$
k_\eps(t)
=
V_0G e^{-tH} \1_{(\eps,\delta)}(H)G^*V_0, 
\quad
k_\eps^{(0)}(t)
=
G e^{tH_0} \1_{(-\delta,-\eps)}(H)G^*,
\quad 
t>0.
$$
First we check that these operators 
are bounded and give an estimate for their norms as $\eps\to 0$:
\begin{lemma}\label{lem6}
We have
\begin{equation}
\label{eq:24}
\sup_{0<\eps\leq\delta} 
\left( \left\| \mathcal Z_\eps \right\| + \norm{ \mathcal Z^{(0)}_\eps} \right) <\infty 
\end{equation}
and 
$$
\norm{\left( \mathcal Z_\eps \right)^* \mathcal Z_\eps}_p^p 
+ 
\norm{\left( \mathcal Z_\eps^{(0)} \right)^* \mathcal Z_\eps^{(0)} }_p^p = O(|\ln\eps|)
\qquad\text{as}\ \eps\to 0 \,.
$$
\end{lemma}

The bound \eqref{eq:24} is already contained in \cite{P1}, but we include a proof for the sake of completeness.

\begin{proof}
By the spectral theorem, we have
\begin{equation}
\label{eq:26}
k_\eps(t)
=
V_0G e^{-tH} \1_{(\eps,\delta)}(H)G^*V_0
= 
\int_{\eps}^{\delta} e^{-t\lambda} \,dF(\lambda) 
= 
\int_{\eps}^{\delta} e^{-t\lambda} F'(\lambda)\,d\lambda \,.
\end{equation}
Thus, the kernel $k_\eps$ is a Laplace transform of an operator valued measure
and so we can apply Lemma~\ref{lem3}. 
From Assumption~\ref{ass1} we know that
$$
\sup_{0\leq\lambda\leq\delta} \|F'(\lambda)\|
\leq
\sup_{0\leq\lambda\leq\delta} \|F'(\lambda)\|_p<\infty \,,
$$
and therefore, by Lemma \ref{lem3} with $q=\infty$, 
we get the uniform boundedness of $\left\| \mathcal Z_\eps \right\|$. 
Further, again by Lemma \ref{lem3} with $q=p$,
$$
\left\| \left( \mathcal Z_\eps \right)^* \mathcal Z_\eps \right\|_p^p 
\leq 
\pi^{p-1} \int_{\eps}^{\delta} \|F'(\lambda)\|_p^p \,\frac{d\lambda}{2\lambda} 
\leq 
C \int_{\eps}^{\delta} \frac{d\lambda}{\lambda} 
= 
O(|\ln\eps|)
\qquad\text{as}\ \eps\to 0 \,.
$$
This proves the lemma for $\mathcal Z_\eps$. The proof for $\mathcal Z_\eps^{(0)}$ is similar
and involves the representation 
\begin{equation}
k_\eps^{(0)}(t)
=
\int_{-\delta}^{-\eps}e^{t\lambda}F_0'(t)dt
=
\int_{\eps}^{\delta}e^{-t\lambda}F_0'(-t)dt.
\label{eq:26a}
\end{equation}
\end{proof}

\begin{lemma}\label{factorization}
The factorisation \eqref{eq:22} holds true:
$$
\1_{(\eps,\delta)}(H)\1_{(-\delta,-\eps)}(H_0)
=
\mathcal Z_\eps  \left(\mathcal Z_\eps^{(0)}\right)^*.
$$
\end{lemma}
\begin{proof}
This is a calculation from \cite{P1}, which we reproduce for completeness.
Let 
$$
L(t)=\1_{(\eps,\delta)}(H)e^{-tH}e^{tH_0}\1_{(-\delta,-\eps)}(H_0).
$$
Then we have
$L(0)=\1_{(\eps,\delta)}(H)\1_{(-\delta,-\eps)}(H_0)$,
$L(+\infty)=0$,
and
\begin{multline*}
L'(t)=-\1_{(\eps,\delta)}(H)e^{-tH}Ve^{tH_0}\1_{(-\delta,-\eps)}(H_0)
\\
=
-(V_0Ge^{-tH}\1_{(\eps,\delta)}(H))^*(Ge^{tH_0}\1_{(-\delta,-\eps)}(H_0)), \quad t>0.
\end{multline*}
Substituting this into 
$$
L(0)-L(+\infty)=-\int_0^\infty L'(t)dt,
$$
and recalling the definition of the operators 
$\mathcal Z_\eps$, $\mathcal Z_\eps^{(0)}$, we obtain the required identity.
\end{proof}

In the next lemma we shall determine the leading order behavior of the Hankel operators 
$\bigl( \mathcal Z_\eps^{(0)}\bigr)^* \mathcal Z_\eps^{(0)}$ and 
$\left( \mathcal Z_\eps \right)^* \mathcal Z_\eps$ 
in $L^2(\R_+,\mathcal K)$
as $\eps\to 0$. 
It turns out that these operators can be approximated in $\mathbf S_p$ norm by 
Hankel operators with the explicit kernels
\begin{equation}
\gamma_\eps(t)F_0'(0), \quad
\gamma_\eps(t)F'(0), \quad t>0,
\label{eq:25a}
\end{equation}
where $\gamma_\eps$ is the model kernel considered in Section~\ref{sec:2}, 
$$
\gamma_\eps(t)
=
\int_\eps^\delta e^{-t\lambda}\,d\lambda.
$$
Identifying $L^2(\R_+,\mathcal K)$ with $L^2(\R_+)\otimes\mathcal K$, 
we shall denote the Hankel operators with the kernels \eqref{eq:25a} by 
$$
\Gamma_\eps \otimes F_0'(0), \quad
\Gamma_\eps \otimes F'(0).
$$
\begin{lemma}\label{lem7}
We have
\begin{align}
\norm{\left( \mathcal Z_\eps^{(0)} \right)^* \mathcal Z_\eps^{(0)} - \Gamma_\eps \otimes F_0'(0)}_p & = O(1)
\qquad\text{as}\ \eps\to 0, 
\label{eq:25b}
\\
\left\| \left( \mathcal Z_\eps \right)^* \mathcal Z_\eps - \Gamma_\eps \otimes F'(0)\right\|_p & = O(1)
\qquad\text{as}\ \eps\to 0.
\label{eq:25c}
\end{align}
\end{lemma}

\begin{proof}
Let us first prove \eqref{eq:25c}. 
Recalling formula \eqref{eq:26} for $k_\eps$, we see that the Hankel operator
$(\mathcal Z_\eps)^* \mathcal Z_\eps - \Gamma_\eps \otimes F'(0)$ 
has the kernel
$$
k_\eps(t)-\gamma_\eps(t) F'(0)
=
\int_\eps^\delta
e^{-t\lambda}(F'(\lambda)-F'(0))d\lambda.
$$
Applying Lemma \ref{lem3} with $q=p$, we get
\begin{equation}
\left\| \left( \mathcal Z_\eps \right)^* \mathcal Z_\eps - \Gamma_\eps \otimes F'(0)\right\|_p^p 
\leq 
\pi^{p-1} \int_\eps^\delta \|F'(\lambda) - F'(0)\|_p^p \,\frac{d\lambda}{2\lambda}.
\label{eq:25d}
\end{equation}
By Assumption~\ref{ass1}, we have $\|F'(\lambda) - F'(0)\|_p=O(\lambda^\varkappa)$, $\lambda\to0$,  
with some $\varkappa>0$. 
It follows that the right side in \eqref{eq:25d} is bounded uniformly in $\eps>0$. 
This proves \eqref{eq:25c}. 
The argument for \eqref{eq:25b}  is similar and involves the representation \eqref{eq:26a} for $k_\eps^{(0)}$. 
\end{proof}

\noindent\textit{Remark.}
Note that Lemma \ref{lem7} together with Lemma \ref{lem2} implies all the assertions in Lemma \ref{lem6}. We have chosen to prove Lemma \ref{lem6} separately for pedagogic reasons, since it does not rely on the machinery to prove Lemma \ref{lem2}.

%%%%%%%%%%%%%%%%%%%%%%%%%%%%%%%%%%%%
\section{Spectral localization}\label{sec:4}
%%%%%%%%%%%%%%%%%%%%%%%%%%%%%%%%%%%%

Let $\Pi^{(1)}_\eps$ and $\wt \Pi^{(1)}_\eps$ be the operators 
defined by \eqref{eq:16a}, \eqref{eq:16}. In this section, we prove

\begin{lemma}\label{lem4}
Let Assumptions~\ref{ass0} and \ref{ass1} hold true. Then for all $q\geq p$ one has
\begin{align}
\norm{\wt \Pi^{(1)}_\eps - \Pi^{(1)}_\eps}_{q}^q &= O(\abs{\ln\eps}^{1/2}) 
\qquad\text{as}\ \eps\to 0 \,.
\label{eq:17aa}
\end{align}
\end{lemma}

\begin{proof}
Setting 
$$
P_\eps=\1_{(\eps,\infty)}(H)\1_{(-\infty,-\eps)}(H_0),
\quad
\wt P_\eps=\1_{(\eps,\delta)}(H)\1_{(-\delta,-\eps)}(H_0),
$$
we can write
$$
\Pi_\eps^{(1)}=P_\eps^* P_\eps, 
\quad
\wt \Pi_\eps^{(1)}=\wt P_\eps^* \wt P_\eps.
$$
First let us estimate the difference $P_\eps-\wt P_\eps$. 
We have
\begin{multline}
P_\eps-\wt P_\eps
=
\bigl(\1_{(\eps,\infty)}(H)\1_{(-\infty,-\eps)}(H_0)
-
\1_{(\eps,\delta)}(H)\1_{(-\infty,-\eps)}(H_0)\bigr)
\\
+
\bigl(\1_{(\eps,\delta)}(H)\1_{(-\infty,-\eps)}(H_0)
-
\1_{(\eps,\delta)}(H)\1_{(-\delta,-\eps)}(H_0)\bigr)
\\
=
\1_{[\delta,\infty)}(H)\1_{(-\infty,-\eps)}(H_0)
-
\1_{(\eps,\delta)}(H)\1_{(-\infty,-\delta]}(H_0).
\label{eq:19}
\end{multline}
Using Lemma~\ref{lem5}, we can estimate separately 
each of the two terms in the 
right side of \eqref{eq:19}:
\begin{gather*}
\norm{\1_{[\delta,\infty)}(H)\1_{(-\infty,-\eps)}(H_0)}_{2p}
\leq
\norm{\1_{[\delta,\infty)}(H)\1_{(-\infty,0)}(H_0)}_{2p}
<\infty,
\\
\norm{\1_{(\eps,\delta)}(H)\1_{(-\infty,-\delta]}(H_0)}_{2p}
\leq
\norm{\1_{(0,\delta)}(H)\1_{(-\infty,-\delta]}(H_0)}_{2p}
<\infty.
\end{gather*}
It follows that 
\begin{equation}
\norm{P_\eps-\wt P_\eps}_{2p}=O(1) \qquad\text{as}\ \eps\to0.
\label{eq:21}
\end{equation}
Next, let us estimate $\wt P_\eps$. 
Since $\norm{\wt P_\eps}\leq1$, we have
\begin{equation}
\norm{\wt P_\eps}_{2p}^{2p}
=
\norm{\abs{\wt P_\eps}^{2p}}_1
\leq
\norm{\abs{\wt P_\eps}^{p}}_1
=
\norm{\wt P_\eps}_{p}^{p}.
\label{eq:21b}
\end{equation}
Recall that by \eqref{eq:22}, we have
$$
\wt P_\eps
=
\mathcal Z_\eps  (\mathcal Z_\eps^{(0)})^*.
$$
Thus, by Lemma~\ref{lem6}, 
$$
\norm{\wt P_\eps}_{p}^{p}
=
\norm{\mathcal Z_\eps  (\mathcal Z_\eps^{(0)})^*}_p^p
\leq
\norm{\mathcal Z_\eps}_{2p}^p
\norm{\mathcal Z_\eps^{(0)}}_{2p}^p
=
O(\abs{\ln \eps}) 
\qquad\text{as}\ \eps \to 0\,.
$$
Combining these formulas, we obtain
\begin{equation}
\norm{\wt P_\eps}_{2p}
=
O(\abs{\ln \eps}^{1/2p}) \qquad\text{as}\ \eps\to0.
\label{eq:20}
\end{equation}
Combining \eqref{eq:21} and \eqref{eq:20}, we also obtain
\begin{equation}
\norm{P_\eps}_{2p}
=
O(\abs{\ln \eps}^{1/2p}) \qquad\text{as}\ \eps\to0\,.
\label{eq:21a}
\end{equation}
Let us prove \eqref{eq:17aa} for $q=p$. We have
$$
\Pi^{(1)}_\eps - \wt \Pi^{(1)}_\eps 
=
P_\eps^*P_\eps
-
\wt P_\eps^*\wt P_\eps
=
(P_\eps^*-\wt P_\eps^*)P_\eps
+
\wt P_\eps^*(P_\eps-\wt P_\eps)\,,
$$
and therefore, by \eqref{eq:21}, \eqref{eq:20}, \eqref{eq:21a}, 
$$
\norm{\Pi^{(1)}_\eps - \wt \Pi^{(1)}_\eps}_p
\leq
\norm{P_\eps^*-\wt P_\eps^*}_{2p}\norm{P_\eps}_{2p}
+
\norm{\wt P_\eps^*}_{2p}\norm{P_\eps-\wt P_\eps}_{2p}
=
O(\abs{\log \eps}^{1/2p})\,, 
$$
as required. 
In order to derive \eqref{eq:17aa} for $q>p$, we use the fact that
$\norm{\wt \Pi^{(1)}_\eps - \Pi^{(1)}_\eps}\leq 2$ and 
argue as in  
\eqref{eq:21b}:
$$
\norm{\wt \Pi^{(1)}_\eps - \Pi^{(1)}_\eps}_{q}^q
=
\norm{\abs{\wt \Pi^{(1)}_\eps - \Pi^{(1)}_\eps}^q}_1
\leq
2^{q-p}
\norm{\abs{\wt \Pi^{(1)}_\eps - \Pi^{(1)}_\eps}^p}_{1}
=
2^{q-p}
\norm{\wt \Pi^{(1)}_\eps - \Pi^{(1)}_\eps}_{p}^p \,.
$$
\end{proof}

\begin{remark}
Note that \eqref{eq:21} immediately implies
$$
\norm{\Pi^{(1)}_\eps - \wt \Pi^{(1)}_\eps}_{2p}=O(1),
$$
which suffices for the proof of Theorem~\ref{thm1} for $f(t)=t^n$ with $n\geq 2p$. 
\end{remark}

%%%%%%%%%%%%%%%%%%%%%%%%%%%%%%%%%%%%%%%%%
\section{Putting it all together}\label{sec:6}
%%%%%%%%%%%%%%%%%%%%%%%%%%%%%%%%%%%%%%%%%

\subsection{Auxiliary statements}
First, we need to relate the scattering matrix to the operators $F_0'$, $F'$. Note that $F_0(\lambda)$ and $F(\lambda)$ are non-decreasing with respect to $\lambda$ and so $F_0'(\lambda)$ and $F'(\lambda)$ are non-negative. In particular, $F_0'(\lambda)^{1/2}$ and $F'(\lambda)^{1/2}$ are well-defined.

Below we frequently use the notation $\approx$ introduced in Section \ref{sec:1.3}.

\begin{lemma}\label{lem9}
Let Assumptions~\ref{ass0}, \ref{ass1} hold true; then 
\begin{equation}
\frac14 |S(0)-I|^2 \approx  \pi^2 F'(0)^{1/2}  F_0'(0) F'(0)^{1/2}
\label{eq:35}
\end{equation}
and 
\begin{equation}
\label{eq:37}
\sum_{\ell=1}^L
a_\ell^{p}<\infty.
\end{equation}
\end{lemma}
\begin{proof}

This is essentially a known statement (see e.g. \cite[Lemma 4]{P1} or \cite[Corollary 4.31]{GKMO}). 
For completeness, we briefly recall the proof. 
First note that by unitarity of $S(0)$ we have
$$
\frac14 |S(0)-I|^2 
=
\frac14 (S(0)^*-I)(S(0)-I)
=
\frac12\re(I-S(0)).
$$
Next, by the stationary representation for the scattering matrix 
(see e.g. \cite[Theorem 5.5.4]{Ya1}), 
the operator $S(0)$ is unitarily equivalent to the operator
(recall that $T(z)$ is defined in \eqref{eq:1a}) 
$$
\wt S(0)=I-2\pi i F_0'(0)^{1/2}(V_0-T(+i0))F_0'(0)^{1/2} \quad 
\text{ in $\mathcal K$.}
$$
It follows that 
\begin{multline*}
\frac12 \re (I-S(0))
\approx
\frac12 \re (I-\wt S(0))
\\
=
\pi \im ( F_0'(0)^{1/2}T(+i0)F_0'(0)^{1/2})
=
\pi^2 F_0'(0)^{1/2}  F'(0) F_0'(0)^{1/2},
\end{multline*}
where we have used \eqref{eq:6c} at the last step.
Denoting $X=F'(0)^{1/2} F_0'(0)^{1/2}$, 
the last operator can be transformed as
$$
\pi^2 F_0'(0)^{1/2}  F'(0) F_0'(0)^{1/2}
=
\pi^2
X^*X
\\
\approx
\pi^2
XX^*
=
\pi^2 F'(0)^{1/2}  F_0'(0) F'(0)^{1/2},
$$
which yields \eqref{eq:35}.
By Assumption~\ref{ass1}, the operator in the right side of \eqref{eq:35} is 
in the class $\mathbf S_{p/2}$.
The relation \eqref{eq:35} implies that the non-zero $a_\ell^2$ coincide with the non-zero eigenvalues of this 
operator; thus, we obtain \eqref{eq:37}.
\end{proof}

\begin{lemma}\label{lem8a}
Let $X_\eps$, $Y_\eps$ be non-negative, compact operators depending on $\eps>0$. 
Assume that for some $q\geq 1$, we have
\begin{equation}
\norm{X_\eps}_q^q=O(\abs{\ln\eps})\,, \qquad
\norm{X_\eps-Y_\eps}_q^q=o(\abs{\ln\eps}) \qquad\text{as}\ \eps\to0.
\label{eq:30b}
\end{equation}
Then 
$$
\Tr Y_\eps^q - \Tr X_\eps^q=o(\abs{\ln \eps}), 
\quad \eps\to0.
$$
\end{lemma}

Of course, we choose the function $\abs{\ln\eps}$ here simply because
this is what comes up in our proof in the next subsection. 

\begin{proof}
We note that for $X\geq0$, we have $\Tr X^q=\norm{X}_q^q$. 
Now the statement of the lemma follows directly from the estimate
$$
\abs{\norm{Y}_q^q-\norm{X}_q^q}
\leq
q\max\{\norm{Y}_q^{q-1}, \norm{X}_q^{q-1}\}\norm{Y-X}_q.
$$
To prove the latter estimate, it suffices to use the elementary inequality
$$
\abs{b^q-a^q}
\leq 
q\max\{b^{q-1}, a^{q-1}\}\abs{b-a},\quad
a\geq0,\quad b\geq0\, ,
$$
with $a=\norm{X}_q$, $b=\norm{Y}_q$ and the inverse triangle inequality
$\abs{\norm{Y}_q-\norm{X}_q}\leq\norm{Y-X}_q$. 
\end{proof}

\subsection{The case $f(t)=t^q$}

\begin{lemma}\label{lem8}
For any $q\geq p$, Theorem~\ref{thm1} holds true with $f(t)=t^q$. That is, 
\begin{equation}
\label{eq:30}
\lim_{\eps\to +0} |\ln\eps|^{-1} \tr \left(\Pi^{(1)}_\eps\right)^q 
= 
\frac{1}{2\pi} \sum_{\ell=1}^L a_\ell^{2q} 
\int_{-\infty}^\infty \frac{dx}{\cosh^{2q}(\pi x)}.
\end{equation}
\end{lemma}

\begin{proof}
Let us denote the operator on the right side of \eqref{eq:35} by $A$,  
\begin{equation}
\label{eq:27}
A = \pi^2 F'(0)^{1/2}  F_0'(0) F'(0)^{1/2} \quad \text{ in $\mathcal K$.}
\end{equation}
It follows from  Lemma~\ref{lem9} that $\{a_\ell^2\}_{\ell=1}^L$ are the non-zero eigenvalues of $A$. 
In the course of the proof we progressively reduce the problem for $\Pi_\eps^{(1)}$ to the 
problem for the following operators:
\begin{align}
\wt \Pi_\eps^{(1)}
&= 
\mathcal Z_\eps^{(0)} \left(\mathcal Z_\eps\right)^* \mathcal Z_\eps  \left(\mathcal Z_\eps^{(0)}\right)^*,
\label{eq:38}
\\
M_{1,\eps}
& =
\mathcal Z_\eps^{(0)} 
(\Gamma_\eps\otimes F'(0))
\bigl(\mathcal Z_\eps^{(0)}\bigr)^*,
\label{eq:39}
\\
M_{2,\eps}
& =
(\Gamma_\eps\otimes F'(0))^{1/2}
\bigl(\mathcal Z_\eps^{(0)}\bigr)^*
\mathcal Z_\eps^{(0)} 
(\Gamma_\eps\otimes F'(0))^{1/2},
\label{eq:40}
\\
\pi^{-2}\Gamma_\eps^2\otimes A
&=
(\Gamma_\eps\otimes F'(0))^{1/2}
(\Gamma_\eps\otimes F_0'(0))
(\Gamma_\eps\otimes F'(0))^{1/2}.
\label{eq:41}
\end{align}
Here formula \eqref{eq:38} follows from the factorization \eqref{eq:22}, 
formulas \eqref{eq:39} and \eqref{eq:40} are the definitions of the 
auxiliary operators $M_{1,\eps}$ and $M_{2,\eps}$, 
and formula \eqref{eq:41} follows from the definition \eqref{eq:27} of
$A$. 
It is convenient to start from the bottom operator \eqref{eq:41} and to move up. 

Denote the right side of \eqref{eq:30} by $\Delta_q$. 
By Lemma \ref{lem2}, we have
\begin{equation}
\label{eq:29}
\tr (\pi^{-2} \Gamma_\eps^2\otimes A)^q
= 
\sum_{\ell=1}^L a_\ell^{2q} \tr (\pi^{-2} \Gamma_\eps^2)^q
= 
|\ln\eps|\ \Delta_q
+o(\abs{\ln\eps}).
\end{equation}
Let us estimate the difference
\begin{equation}
M_{2,\eps}
-
\pi^{-2}\Gamma_\eps^2\otimes A
=
(\Gamma_\eps\otimes F'(0))^{1/2}
\bigl(
{(\mathcal Z_\eps^{(0)})}^*
\mathcal Z_\eps^{(0)} 
-\Gamma_\eps\otimes F_0'(0)
\bigr)
(\Gamma_\eps\otimes F'(0))^{1/2}.
\label{eq:42}
\end{equation}
By Lemma \ref{lem2} and Lemma \ref{lem7},
\begin{align}
\left\| M_{2,\eps} - \pi^{-2} \Gamma_\eps^2\otimes A \right\|_{q} 
& \leq
\left\| M_{2,\eps} - \pi^{-2} \Gamma_\eps^2\otimes A \right\|_{p} 
\notag \\
& \leq 
\left\| \left(\Gamma_\eps\otimes F'(0)\right)^{1/2} \right\|^2  \left\| \left( \mathcal Z_\eps^{(0)} \right)^* \mathcal Z_\eps^{(0)} - \Gamma_\eps\otimes F_0'(0) \right\|_p 
\notag \\
& = \|\Gamma_\eps\|^2 \|F'(0)\|  \left\| \left( \mathcal Z_\eps^{(0)} \right)^* \mathcal Z_\eps^{(0)} - \Gamma_\eps\otimes F_0'(0) \right\|_p 
= O(1) \,.
\label{eq:31a}
\end{align}
Let us apply Lemma~\ref{lem8a} with $X_\eps=\pi^{-2} \Gamma_\eps^2\otimes A$ and 
$Y_\eps=M_{2,\eps}$. 
In the hypothesis \eqref{eq:30b} of this lemma, 
the first estimate follows from \eqref{eq:29} and the second estimate holds by \eqref{eq:31a}. 
We obtain
\begin{equation}
\label{eq:31}
\lim_{\eps\to +0} |\ln\eps|^{-1} \tr \left( M_{2,\eps} \right)^q = \Delta_q\,.
\end{equation}
Next, by definitions \eqref{eq:39} and \eqref{eq:40}, we have 
$M_{2,\eps}\approx M_{1,\eps}$, and therefore \eqref{eq:31} yields
$$
\lim_{\eps\to 0+} |\ln\eps|^{-1} \tr \left( M_{1,\eps} \right)^q = \Delta_q\,.
$$
Further, similarly to \eqref{eq:42}, \eqref{eq:31a},
$$
\wt \Pi_\eps^{(1)}-M_{1,\eps}
=
\mathcal Z_\eps^{(0)}
(\mathcal Z_\eps^*\mathcal Z_\eps-\Gamma_\eps\otimes F'(0))
{(\mathcal Z_\eps^{(0)})}^*,
$$
and so by Lemmas~\ref{lem6} and \ref{lem7}
$$
\norm{\wt \Pi_\eps^{(1)}-M_{1,\eps}}_q
\leq
\norm{\wt \Pi_\eps^{(1)}-M_{1,\eps}}_p
\\
\leq
\norm{\mathcal Z_\eps^{(0)}}^2
\norm{\mathcal Z_\eps^*\mathcal Z_\eps-\Gamma_\eps\otimes F'(0)}_p
=
O(1) \,.
$$
Now we can apply Lemma~\ref{lem8a} with $X_\eps=M_{1,\eps}$ and 
$Y_\eps=\wt \Pi_\eps^{(1)}$, which yields
$$
\lim_{\eps\to 0+} |\ln\eps|^{-1} \tr ( \wt \Pi^{(1)}_\eps )^q
=\Delta_q\,.
$$
Finally, we apply Lemma~\ref{lem8a} once again 
with $X_\eps={\wt \Pi}^{(1)}_\eps$ and $Y_\eps=\Pi^{(1)}_\eps$. 
The second estimate in the hypothesis \eqref{eq:30b} is given by Lemma~\ref{lem4}.
This yields 
$$
\lim_{\eps\to 0+} |\ln\eps|^{-1} \tr \left( \Pi^{(1)}_\eps \right)^q = \Delta_q\,,
$$
as required.
\end{proof}

\subsection{Proofs of Theorem~\ref{thm1} and Corollary \ref{cor2}}

\begin{proof}[Proof of Theorem \ref{thm1}]
First consider the operator $\Pi^{(1)}_\eps$. 
By Lemma \ref{lem8}, the asymptotics \eqref{eq:6} holds if $f$ is $t^q$ times a polynomial.
The proof for a general $f$ follows from the Weierstrass approximation theorem as at the end of the proof of Lemma \ref{lem2}.

Next, consider the operator $\Pi^{(2)}_\eps$. 
By our assumptions on the functions $\psi_\eps^{\pm}$, we have
\begin{align*}
\Pi_\eps^{(2)}
& =
\psi_\eps^{-}(H_0)
\psi_\eps^{+}(H)
\psi_\eps^{-}(H_0)
\leq
\psi_\eps^{-}(H_0)
\1_{(\eps,\infty)}(H)
\psi_\eps^{-}(H_0)
\\
& \approx
\1_{(\eps,\infty)}(H)
(\psi_\eps^{-}(H_0))^2
\1_{(\eps,\infty)}(H)
\leq
\1_{(\eps,\infty)}(H)
\1_{(-\infty,-\eps)}(H_0)
\1_{(\eps,\infty)}(H)
\approx
\Pi_\eps^{(1)}\,.
\end{align*}
By the min-max principle, it follows that 
\begin{equation}
\tr(\Pi_\eps^{(2)})^q\leq\tr(\Pi_\eps^{(1)})^q\,, \quad q\geq p\,.
\label{eq:33}
\end{equation}
Similarly, we have a lower bound 
\begin{multline*}
\Pi_\eps^{(2)}
=
\psi_\eps^{-}(H_0)
\psi_\eps^{+}(H)
\psi_\eps^{-}(H_0)
\geq
\psi_\eps^{-}(H_0)
\1_{(2\eps,\infty)}(H)
\psi_\eps^{-}(H_0)
\\ 
\approx
\1_{(2\eps,\infty)}(H)
(\psi_\eps^{-}(H_0))^2
\1_{(2\eps,\infty)}(H)
\leq
\1_{(2\eps,\infty)}(H)
\1_{(-\infty,-2\eps)}(H_0)
\1_{(2\eps,\infty)}(H)
\approx
\Pi_{2\eps}^{(1)}\,, 
\end{multline*}
and therefore 
\begin{equation}
\tr(\Pi_\eps^{(2)})^q\geq\tr(\Pi_{2\eps}^{(1)})^q\,, \quad q\geq p\,.
\label{eq:34}
\end{equation}
A combination \eqref{eq:33} and \eqref{eq:34} gives the analogue of Lemma~\ref{lem8} for 
$\Pi_\eps^{(2)}$ and so again we obtain the required statement by application of the
Weierstrass approximation theorem. 
\end{proof}

\begin{proof}[Proof of Corollary \ref{cor2}]
Let $f(x)$ be a continuous function on $[0,1]$ such that $f$ vanishes
in a neighbourhood of zero and $\ln(1-x)\leq f(x)$. 
We have
$$
\ln\det\left(I-\Pi_\eps^{(j)}\right) = \tr \ln (I-\Pi_\eps^{(j)}) \leq \tr f(\Pi_\eps^{(j)}) \,,
$$
and therefore by Theorem \ref{thm1},
$$
\limsup_{\eps\to 0+} |\ln\eps|^{-1} \ln\det\left( I-\Pi_\eps^{(j)}\right) 
\leq 
\frac{1}{2\pi} \sum_{\ell=1}^L \int_{-\infty}^\infty f \left( \frac{a_\ell^2}{\cosh^2(\pi x)} \right) dx \,.
$$
Taking the infimum over all such $f$ in the right side, we obtain the required statement. 
\end{proof}

%%%%%%%%%%%%%%%%%%%%%%%%%%%%%

\subsection*{Acknowledgments} Financial support from the U.S.~National Science Foundation through grants PHY-1347399 and DMS-1363432 (R.F.) is acknowledged. The authors are grateful to the anonymous referee for helpful suggestions.

%%%%%%%%%%%%%%%%%%%%%%%%%%%%%%%%%%%%%%%%%%%%%%%%%%%%%%%%%%%%%%%%%%%%%%%%%%%%%%%%
%%%%%%%%%%%

\bibliographystyle{amsalpha}

\end{document}